\DeclareSymbolFont{AMSb}{U}{msb}{m}{n}    %%DELETE IF YOU REMOVE CHARTER FONT 
\documentclass[noamsfonts]{amsart}        %%PUT BACK IN IF YOU REMOVE CHARTER FONT

\usepackage{amsmath}
\usepackage{hyperref}
\usepackage{microtype}
\usepackage{paralist}
\usepackage{thmtools}
\usepackage{tikz}
\usepackage{tikz-cd}
\usepackage[all,cmtip]{xy}
\usetikzlibrary{matrix, arrows}

\usepackage[textsize=footnotesize, colorinlistoftodos]{todonotes}
\makeatletter
\providecommand\@dotsep{5}
\makeatother

\definecolor{my-linkcolor}{rgb}{0.75,0,0}
\definecolor{my-citecolor}{rgb}{0,0.5,0}
\definecolor{my-urlcolor}{rgb}{0,0,0.75}
\hypersetup{
    colorlinks, 
    linkcolor={my-linkcolor},
    citecolor={my-citecolor}, 
    urlcolor={my-urlcolor}
}

\usepackage[charter]{mathdesign}   %%DELETE IF YOU REMOVE CHARTER FONT 

\newcommand{\sslash}{\mathbin{\mathchoice{/\mkern-6mu/}% \displaystyle
{/\mkern-6mu/}% \textstyle                                  
{/\mkern-5mu/}% \scriptstyle
{/\mkern-5mu/}}}% \scriptscriptstyle

\DeclareMathOperator{\bbC}{\mathbb{C}}

\DeclareMathOperator{\calD}{\mathcal{D}}

\DeclareMathOperator{\calO}{\mathcal{O}}

\DeclareMathOperator{\frakg}{\mathfrak{g}}
\DeclareMathOperator{\frakh}{\mathfrak{h}}

\DeclareMathOperator{\ann}{Ann}
\DeclareMathOperator{\bl}{Bl}
\DeclareMathOperator{\der}{Der}
\DeclareMathOperator{\gr}{gr}
\DeclareMathOperator{\hc}{HC}

\DeclareMathOperator{\rad}{Rad}
\DeclareMathOperator{\rk}{rk}

\DeclareMathOperator*{\lcm}{lcm}

\declaretheorem[parent=section]{theorem}
\declaretheorem[unnumbered, name=Theorem]{theorem*}
\declaretheorem[sibling=theorem]{proposition}
\declaretheorem[unnumbered, name=Proposition]{proposition*}
\declaretheorem[sibling=theorem]{conjecture}
\declaretheorem[unnumbered, name=Conjecture]{conjecture*}
\declaretheorem[sibling=theorem, style=definition]{definition}
\declaretheorem[unnumbered, style=definition, name=Definition]{definition*}
\declaretheorem[sibling=theorem]{lemma}
\declaretheorem[unnumbered, name=Lemma]{lemma*}

\declaretheorem[unnumbered, style=remark, name=Remark]{remark*}

\declaretheorem[unnumbered, style=remark, name=Assumption]{assumption*}
\declaretheorem[sibling=theorem]{corollary}
\declaretheorem[unnumbered , name=Corollary]{corollary*}

\title{The Bernstein-Sato $b$-function of the Vandermonde Determinant}
\author{Asilata Bapat and Robin Walters}
\address[Asilata Bapat]{Department of Mathematics, The University of Chicago, Chicago, IL 60637}
\email{asilata@math.uchicago.edu}

\address[Robin Walters]{Department of Mathematics, The University of Chicago, Chicago, IL 60637}
\email{robin@math.uchicago.edu}

\begin{document}
\begin{abstract}
  The Bernstein-Sato polynomial, or the $b$-function, is an important invariant of singularities of hypersurfaces that is difficult to compute in general.
  We describe a few different results towards computing the $b$-function of the Vandermonde determinant $\xi$.

  We use a result of Opdam to produce a lower bound for the $b$-function of $\xi$.
  This bound proves a conjecture of Budur, Musta{\c{t}}{\u{a}}, and Teitler for the case of finite Coxeter hyperplane arrangements, proving the Strong Monodromy Conjecture in this case.

  In our second set of results, we show the duality of two $\calD$-modules, and conclude that the roots of the $b$-function of $\xi$ are symmetric about $-1$.
  We then use some results about jumping coefficients to prove an upper bound for the $b$-function of $\xi$, and finally we conjecture a formula for the $b$-function of $\xi$.
\end{abstract}
\maketitle

\section{Introduction}

The Bernstein-Sato polynomial, also called the $b$-function, is a relatively fine invariant of singularities of hypersurfaces.
Let $f$ be a polynomial function on an affine space $X$, and let $\calD_X$ be the ring of differential operators on $X$.
Then the $b$-function of $f$ can be defined as the minimal polynomial $b_f(s)$ for the operator $s$ on the holonomic $\calD_X[s]$-module $\calD_X[s]f^s / \calD_X[s]f^{s+1}$ \cite{kas:76}.  
Computing $b_f(s)$ for general $f$ remains a very difficult problem, thus research has focused on certain special cases.
The goal of this paper is to prove some results about the $b$-function of the Vandermonde determinant.

Sato's original case of study for $b$-functions was semi-invariants of group actions on prehomogeneous vector spaces \cite{sat:90,sat.shi:74}. 
It remains true that many of the computed examples of $b$-functions are for invariant or semi-invariant functions $f$.

Another important case is that of $f$ defining a hyperplane arrangement. 
This active area of research draws together the study of $b$-functions with several other important singularity invariants, including Igusa zeta functions, jumping coefficients, and local monodromy \cite{sai:06,sai:07,wal:05,bud.mus.tei:11,bud.sai:10,bud:12}.

Let $G$ be a complex connected reductive Lie group with Lie algebra $\frakg$.
Let $\frakh\subset \frakg$ be a Cartan subalgebra, and let $R \subset \frakh^*$ be the associated root system with Weyl group $W$.
Define $\xi$ to be the product of the positive roots:
\[
\xi = \prod_{\alpha \in R^{+}} \alpha.
\]
The function $\xi$ is anti-symmetric with respect to the $W$-action on $\frakh$ and is the Jacobian determinant of the quotient map $\frakh \to \frakh/W$. 
Its zero locus $V(\xi)$ is a union of hyperplanes, consisting of points fixed by at least one non-trivial element of $W$.  
Thus $V(\xi)$ is the complement of $\frakh^\mathrm{reg}$.  
The $W$-invariant function $\xi^2$ is called the \emph{discriminant} of the root system $R$.
Let $\Delta$ denote the pullback of $\xi^2$ under the Chevalley isomorphism $\bbC[\frakg]^G \overset{\cong}{\longrightarrow} \bbC[\frakh]^W$.

  For a root system of type $A_{n-1}$, we denote $\xi$ by $\xi_n$.
  This polynomial is recognized as the \emph{Vandermonde determinant}: 
  \[
  \xi_n = \prod_{1 \leq i < j \leq n} (x_i - x_j).
  \]
  In this case, $\Delta$ sends a matrix in $\frakg$ to the discriminant of its characteristic polynomial.

For any $S\subset \{1,\dots, n\}$, define the following:
\[
\xi_S = \prod_{\{(i,j)\in S\times S\mid i < j\}} (x_i - x_j).
\]
If $P = \{P_1,\dots, P_r\}$ is a set partition of $\{1,\dots, n\}$, define $\xi_P = \prod_{t=1}^r\xi_{P_i}$.
Given an integer partition $\lambda = (\lambda_1\geq \dots \geq \lambda_r)$ of $n$, denoted by $\lambda \vdash n$, consider the set partition $P$ of $\{1,\dots, n\}$ in which $P_1 = \{1,\dots, \lambda_1\}$, $P_2 = \{\lambda_1+1,\ldots, \lambda_1+\lambda_2\}$, and so on.
We will write $\xi_\lambda$ to mean $\xi_P$ for the above set partition $P$.

We conjecture the following formula for the $b$-function of $\xi_n$.
%%%%%%%%%%%%%%%%%%%%%%%%%%%%%%%%%%%%%%%%%
% main conjecture
\begin{conjecture}[Main Conjecture]
  \label{conj:main-conjecture}
  The $b$-function of $\xi_n$ is given by the following recursive formula:
  \[
  b_{\xi_n}(s) = \lcm_{\substack{\lambda\vdash n\\ \lambda \neq (n)}}\left(b_{\xi_\lambda}(s)\right)\cdot \prod_{i=(n-1)}^{(n-1)^2}\left(s + \frac{i}{\binom{n}{2}} \right).
  \]
\end{conjecture}

Let $X= \bbC^n$, so that $V(\xi_n)\subset X$.
Our most significant result towards proving this conjecture is the following symmetry property.
\begin{theorem}
  \label{thm:d-module-dual}
  The $\calD_{X}[s]$-module $\calD_{X}[s]\xi_n^s$ is dual to the module $\calD_{X}[s]\xi_n^{-s-1}$ under the Verdier duality functor. 
\end{theorem}
We deduce the following.
\begin{corollary}\label{cor:symmetry-b}
  The roots of the $b$-function of $\xi_n$ are symmetric about the point $-1$.
  In other words, $b_{\xi_n}(s) = \pm b_{\xi_n}(-s-2)$.
\end{corollary}

In the case of any root system, the function $\xi^2$ is $W$-invariant, and we can consider its image in $\bbC[\frakh/W]$.
That is, consider $\xi^2$ as a polynomial in coordinates given by homogeneous basic invariant polynomials $e_i$.
For clarity, denote this polynomial by $g_n$ so that $g_n(e_1,\dots,e_n) = \xi_n^2(x_1,\dots,x_n)$. 
In \cite{opd:89}, Eric Opdam found the $b$-function for $g_n$.   
In general, $b_{g_n}(s)$ divides $b_{\xi^2}(s)$, but usually falls far short of equality. 
Moreoever, for a general $f$, it is always true that $b_f(2s+1)b_f(2s)\mid b_{f^2}(s)$, but equality does not always hold.

In \autoref{sec:budur-hc-etc}, we are able to exploit the relationship between $\xi_n$ and $g_n$ to prove a conjecture of Budur, Musta{\c{t}}{\u{a}}, and Teitler \cite[Conjecture 1.2]{bud.mus.tei:11} for the case of reduced finite Coxeter arrangements:

\begin{theorem}
  \label{thm:n-over-d}
  Let $\frakh$ be a Cartan subalgebra of a simple complex Lie algebra $\frakg$.
  Let $\xi\in \bbC[\frakh]$ be the product of the positive roots as defined earlier.
  Let $d = \deg(f)$ and let $n = \dim(\frakh)$.
  Then $-n/d$ is always a root of the $b$-function of $\xi$.
\end{theorem}

Budur, Musta{\c{t}}{\u{a}}, and Teitler prove a conjecture called the Weak Monodromy Conjecture \cite[Theorem 1.3(a)]{bud.mus.tei:11}, which shows that poles of zeta functions correspond to eigenvalues of Milnor monodromy. 
They further state the Strong Monodromy Conjecture, which claims that roots of $b$-functions correspond to eigenvalues of Milnor monodromy, and which implies the Weak Monodromy Conjecture. 
In \cite[Theorem 1.3(b)]{bud.mus.tei:11}, they reduce the strong conjecture to the so-called $n/d$ conjecture.  
Our result, \autoref{thm:n-over-d}, is the $n/d$ conjecture in the case of reduced finite Coxeter arrangements, and thus proves the Strong Monodromy Conjecture in this case.

%%%%%%%%%%%%%%%%%%%%%%%%%%%%%%%%%%%%%%%%%%

\subsection*{Acknowledgments}
We are extremely grateful to Nero Budur for suggesting part of this problem and sharing some of his ideas with us.
We are grateful to Uli Walther for sending us his recent draft \cite{wal:15} about related problems.
We are indebted to our advisor Victor Ginzburg for his constant support, ideas, and enthusiasm.
The second author thanks Andr{\'a}s~L{\H{o}}rincz for suggesting the proof of \autoref{prop:lorincz}.
We also thank A.~Beilinson, A.~Deopurkar, M.~Emerton, M.~Musta{\c{t}}{\u{a}}, G.~Williamson, and D.~Ben-Zvi for helpful conversations.

\section{Relationship to \texorpdfstring{$b$}{b}-functions for \texorpdfstring{$\bbC[\frakg]$}{C[g]} }
\label{sec:budur-hc-etc}

%%%%%%%%%%%%%%%%%%%%%%%%%%%%%%%%%%%%%%%%%%%%%%
%formula relating b-functions on g to h

Let $\frakg$, $\frakh$, $W$, and $\xi$ be as defined in the introduction.
By the Chevalley-Shephard-Todd theorem, $\frakh/W$ is an $n$-dimensional affine space (where $n = \rk(G) = \dim(\frakh)$), and hence $\bbC[\frakh/W]$ is a polynomial ring in $n$ variables.
Fix a homogeneous free set of generators for this polynomial ring, so that $\bbC[\frakh/W] = \bbC[e_1,\dots, e_n]$.
From now on, we write $\bbC[\frakh/W]$ to mean polynomials in the generators $\{e_1,\dots,e_n\}$, and $\bbC[\frakh]^W$ to mean polynomials in the generators of $\bbC[\frakh]$.
The polynomial $\xi^2$ is $W$-invariant, and so we let $g$ be the corresponding polynomial in $\bbC[\frakh/W]$.

In the paper \cite{opd:89}, Opdam computed the $b$-function of $g$.
We prove the following relationship between the $b$-functions of $g$ and of $\xi$.
\begin{theorem}
  \label{thm:budur}
  The function $b_g(s)$ divides the function $b_{\xi}(2s+1)$.
\end{theorem}

%%%%%%%%%%%%%%%%%%%%%%%%%%%%%%%%%%%%%%%%%%%%%%
%Proof of budur conjecture

\begin{proof}
  The inclusion map $\frakh\hookrightarrow \frakg$ induces a restriction map $\rho\colon \bbC[\frakg]^G\to \bbC[\frakh]^W$, which is an isomorphism by the Chevalley restriction theorem.
  Let $\Delta = \rho^*(\xi^2)$, which is an element of $\bbC[\frakg]^G$.

  Let $L_{\xi^2}(s)\in D(\frakh)[s]$ be an operator that satisfies $L_{\xi^2}(s)(\xi^{2(s+1)}) = b_{\xi^2}(s)\cdot (\xi^2)^s$. 
  Since $\xi^2$ is $W$-invariant, we may assume (by averaging) that $L_{\xi^2}(s)\in D(\frakh)^W[s]$.

  The space $D(\frakh)^W$ of $W$-invariant operators acts on $\bbC[\frakh/W]$, by pulling back via the isomorphism $\bbC[\frakh/W]\cong \bbC[\frakh]^W$.
  For any $L\in D(\frakh)^W$, let $\varphi(L)$ be the corresponding differential operator in $D(\frakh/W)$.
  Clearly, $\varphi$ extends to a map $\varphi\colon D(\frakh)^W[s]\to D(\frakh/W)[s]$.
  Applying $\varphi$ to $L_{\xi^2}(s)$, we see that $\varphi(L_{\xi^2}(s))(g^{s+1}) = b_{\xi^2}(s)\cdot g^s$.

  This equation shows that the minimal $b$-function of $g$ divides $b_{\xi^2}(s)$, that is,
  \begin{equation}
    \label{eq:g-xi-n}
    b_g(s)\mid b_{\xi^2}(s).
  \end{equation}
  
  Similarly, we have a map $D(\frakg)^G[s]\to D(\frakg\sslash G)[s]$.
  Let $L_\Delta(s)$ be an operator that satisfies $L_\Delta(s)(\Delta^{s+1}) = b_\Delta(s)\cdot \Delta^s$.
  Since the action of $G$ on $D(\frakg)[s]$ is locally finite, we may assume by averaging that $L_\Delta(s)\in D(\frakg)^G[s]$.
  By a similar argument as above for the quotient $\frakg\to \frakg\sslash G$ instead of $\frakh\to \frakh/W$, we see that
  \begin{equation}
    \label{eq:g-delta}
    b_g(s)\mid b_\Delta(s).
  \end{equation}
  
  Let $L_{\xi}(s)\in D(\frakh)$ such that $L_{\xi}(s)(\xi^{s+1}) = b_{\xi}(s)\cdot \xi^s$.
  Observe that
  \[
  L_{\xi}(2s)L_{\xi}(2s+1)(\xi^{2(s+1)}) = b_{\xi}(2s)b_{\xi}(2s+1)\cdot (\xi^2)^s.
  \]
  Therefore the minimal $b$-function of $\xi^2$ divides $b_{\xi}(2s)b_{\xi}(2s+1)$, that is,
  \begin{equation}
    \label{eq:xi-n-and-square}
    b_{\xi^2}(s)\mid b_{\xi}(2s)b_{\xi}(2s+1).
  \end{equation}
  From \eqref{eq:g-xi-n} and \eqref{eq:xi-n-and-square}, we see that
  \begin{equation}
    \label{eq:g-and-xi-1}
    b_g(s) \mid b_{\xi}(2s)b_{\xi}(2s+1).
  \end{equation}

  We use the following theorem from Section 4 of \cite{har:64}.
  \begin{proposition}[Harish-Chandra]
    There is a homomorphism of algebras $\hc\colon D(\frakg)^G\to D(\frakh)^W$, called the \emph{Harish-Chandra homomorphism}.
    % For any $P\in D(\frakg)^G$, $\hc(P)$ is defined to be $\xi\circ \rad(P)\circ \xi^{-1}$, where $\rad(P)$ is the radial part of $P$.
  \end{proposition}
  
  Clearly, $\hc$ extends to a map $\hc\colon D(\frakg)^G[s]\to D(\frakh)^W[s]$.
  Recall that $L_\Delta(s)$ is in $D(\frakg)^G[s]$, and was chosen such that $L_\Delta(s)(\Delta^{s+1}) = b_\Delta(s)\cdot \Delta^s$.
  As explained in \cite{har:64}, the map $HC$ is the conjugation by $\xi$ of the radial part map $\rad$.
    Since $\Delta$ corresponds to the function $\xi^2$ under the Chevalley restriction map, we have
  \begin{align*}
    \hc(L_\Delta(s))\cdot (\xi^2)^{s+1} &= \xi \circ \rad(L_\Delta(s))\circ \xi^{-1}(\xi^2)^{s+1}\\
                                        &= \xi \circ \rad(L_\Delta(s))(\xi^2)^{(2s+1)/2}\\
                                        &= \xi \cdot b_{\Delta}(s-1/2)\cdot (\xi^2)^{(2s-1)/2}\\
                                        &= b_{\Delta}(s-1/2)\cdot \xi^{2s},
  \end{align*}
  which shows that $b_{\xi^2}(s)\mid b_{\Delta}(s-1/2)$.

  Results of \cite{wal:93} and \cite{lev.sta:95} show that $\hc$ is surjective, and hence $L_{\xi^2}(s)\in D(\frakh)^W$ can be lifted to an operator in $D(\frakg)^G$.
  By running the previous argument in the reverse direction, we can see that $b_\Delta(s-1/2)\mid b_{\xi^2}(s)$.
  We conclude that $b_{\xi^2}(s) = b_\Delta(s - 1/2)$, and by changing variables that
  \begin{equation}
    \label{eq:xi2-and-delta}
    b_{\xi^2}(s + 1/2) = b_\Delta(s).
  \end{equation}

  From \eqref{eq:g-delta}, \eqref{eq:xi-n-and-square}, and \eqref{eq:xi2-and-delta}, we see that
  \begin{equation}
    \label{eq:g-and-xi-2}
    b_g(s)\mid b_\xi(2s+1)b_\xi(2s+2).
  \end{equation}

  Suppose that $b_g(s)\nmid b_\xi(2s+1)$.
  This means that there is some $c$ that is a root of $b_g(s)$ of some multiplicity $m$, but is a root of $b_\xi(2s+1)$ of multiplicity $k < m$ (where $k$ may be zero).
  By \eqref{eq:g-and-xi-1}, $c$ must be a root of $b_\xi(2s)$, and by \eqref{eq:g-and-xi-2}, $c$ must be a root of $b_\xi(2s+2)$.
  
  By \cite[Theorem 1]{sai:06}, the difference between any two roots of the $b$-function of $f$, a hyperplane arrangement, is less than $2$.
  So $c$ cannot be a root of both $b_\xi(2s)$ and $b_\xi(2s+2)$, and we have a contradiction.
  This argument proves that $b_g(s)\mid b_\xi(2s+1)$.
  This argument proves that $b_g(s)\mid b_\xi(2s+1)$.
\end{proof}

The proof of the $n/d$ conjecture for finite Coxeter arrangements now follows quite easily.
\begin{proof}[Proof of \autoref{thm:n-over-d}]
  Let $d_1\leq \dots \leq d_n$ be a list of the degrees of the fundamental invariants of the Lie group $G$.
  The degree of the highest fundamental invariant is equal to the Coxeter number.
  Recall that $n$ is the rank of the root system, and the total number of roots equals $2d$.
  It is known (see, e.g., \cite[Section 3.18]{hum:90}) that $d_n\cdot n = 2d$.

  From \cite{opd:89}, we know that
  \[
  b_g(s) = \prod_{i=1}^n\prod_{j=1}^{d_i-1}\left( s+\frac{1}{2} + \frac{j}{d_i}\right).
  \]
  Notice that one of the factors above is
  \[
  \left(s+\frac{1}{2}+\frac{1}{d_n}\right) = \left(s + \frac{1}{2} + \frac{n}{2d}\right).
  \]
  So $-(1/2 + n/(2d))$ is a root of $b_g(s)$ and hence of $b_\xi(2s+1)$, which precisely means that $b_\xi(-n/d) = 0$.
\end{proof}

\section{Symmetry of the roots of the \texorpdfstring{$b$}{b}-function}
In this section we discuss the proofs of \autoref{thm:d-module-dual} and \autoref{cor:symmetry-b}, namely the duality of the $D$-modules corresponding to $\xi_n^s$ and $\xi_n^{-s-1}$ in type $A_n$, and the symmetry of the roots of Bernstein-Sato polynomial of $\xi_n$ around the point $-1$.
\autoref{cor:symmetry-b} is a special case of Theorem 4.1 of \cite{mac:12}.
It follows quite easily from \autoref{thm:d-module-dual}, and this proof is explained in Section 4 of the above paper.
It remains to prove \autoref{thm:d-module-dual}, and we begin by recalling some definitions.
\begin{definition}[K.~Saito, \cite{sai:80}]
  A divisor $F$ on a space $X$ is called a \emph{free divisor} if the logarithmic vector fields $\der_{\bbC}(-\log F)$ form a locally free $\calO_X$-module.
\end{definition}

Let $X$ be an affine space.
Under the order filtration on $\calD_X$, we have $\gr\calD_X \cong \bbC[T^*X]$.
Recall that for any $L\in \calD_X[s]$ of order $k$, its \emph{principal symbol}, denoted $\sigma(L)$, is its image in the $k$th graded piece of $\bbC[T^*X][s]$.
\begin{definition}
  Let $F$ be a free divisor in an affine space $X$ and let $\delta_1,\dots,\delta_n$ be a basis of $\der_{\bbC}(-\log F)$. 
  Let $h\in \calO_{X,p}$ be a reduced equation of $F$, and moreover suppose that $\delta_i(h) = \alpha_ih$.
  Then $F$ is called \emph{strongly Koszul} if the sequence of principal symbols $\sigma(\delta_1-\alpha_1s), \dots, \sigma(\delta_n-\alpha_ns), \sigma(h)$ is a regular sequence in $\bbC[T^*X][s]$.
\end{definition}

The statements of Proposition 2.1, Proposition 2.2, and Theorem 4.1 from the paper \cite{mac:12} together imply the following corollary.
\begin{corollary}[From \cite{mac:12}]
  \label{cor:macarro-summary}
  Let $h$ be non-constant reduced germ of a holomorphic function such that the divisor $V(h)$ is a strongly Koszul free divisor.
  Then the $\calD_X[s]$-module $\calD_X[s]h^s$ is isomorphic to the Verdier dual of the $\calD$-module $\calD_X[s]h^{-s-1}$.
\end{corollary}

Hence to prove \autoref{thm:d-module-dual}, it is sufficient to show that $V(\xi)$ is a strongly Koszul free divisor. 
The following proposition is known.
\begin{proposition}[Proposition 2 of \cite{ter:80}]
  \label{thm:coxeter-arrangements}
  Let $F\subset X \cong \bbC^n$ be any finite Coxeter arrangement, with corresponding Coxeter group $W$.
  Let $\underline{x} = (x_1,\dots,x_n)$ be coordinates on $X$.
  Let $\{e_1,\dots, e_n\}$ be a set of homogeneous free polynomial generators of the ring of invariants $\bbC[X]^W$.
  \begin{enumerate}
  \item The hyperplane arrangement $F$ is a free divisor.
  \item For each $i$, the gradient vector field
    \[
    \delta_i = \sum_{j=1}^n\frac{de_i}{dx_j}\partial_j
    \]
    is a logarithmic vector field.
  \item The vector fields $\delta_1,\dots,\delta_n$ form a free $\calO_X$-basis of $\der_{\bbC}(-\log F)$.
  \end{enumerate}
\end{proposition}

We conjecture the following, from which \autoref{thm:d-module-dual} would follow for all finite Coxeter arrangements.
\begin{conjecture}
  Let $F\subset \bbC^n$ be a finite Coxeter arrangement.
  Then $F$ is a strongly Koszul free divisor.
\end{conjecture}

From now on, we will only work in type $A_n$, and we now prove the above conjecture for $A_n$ arrangements.
Consider the (scaled) power-sum basis of the ring of symmetric polynomials:
\[
e_i = \frac{1}{i}\cdot (x_1^i + \dots + x_n^i),
\]
for $1\leq i \leq n$.
Set $\delta_i$ to be the gradient of $e_i$, namely,
\[
\delta_i = \sum_{j=1}^n\frac{de_i}{dx_j}\partial_j= x_1^{i-1}\partial_1 + \dots + x_n^{i-1}\partial_n.
\]
Then $\{\delta_1,\dots,\delta_n\}$ is a basis for the logarithmic vector fields.

By \autoref{thm:coxeter-arrangements}, each $\delta_i$ is a logarithmic vector field.
So we have $\delta_i(\xi^s) = s\alpha_i\xi^s$,
where $\alpha_i$ is some polynomial in the variables $\{x_1,\dots,x_n\}$.
In fact, direct computation shows that 
\begin{equation}\label{eqn:alphak}
\alpha_k = \sum_{i<j}\left(\frac{x_i^{k-1}-x_j^{k-1}}{x_i-x_j}\right).
\end{equation}
For each $i$, we have $(\delta_i - s\alpha_i)\in \ann(f^s)$.
Set $\gamma_i = \sigma(\delta_i - s\alpha_i)$ in $\bbC[T^*X]$ to be its principal symbol.

The next two lemmas contain most of the work of the proof of \autoref{thm:d-module-dual}.
\begin{lemma}
  \label{lem:regularity-gammas}
  The sequence $(\gamma_1,\dots, \gamma_n)$ is regular in $\bbC[T^*X][s]$.
\end{lemma}

\begin{proof}

  We use the coordinates $(x_1,\dots, x_n,\partial_1,\dots,\partial_n)$ on $\bbC[T^*X]$.
Recall that
\[
\gamma_i = x_1^{i-1}\partial_1 + \dots + x_n^{i-1}\partial_n - \alpha_n s.
\]
We encode the data of the sequence $(\gamma_1,\ldots,\gamma_n)$ in the following $n\times (n+1)$ matrix:
\[
\Gamma =
\begin{pmatrix}
  1&\dots &1 &-\alpha_1\\
  x_1&\dots &x_n&-\alpha_2\\
  \vdots & \ddots & \vdots &\vdots\\
  x_1^{n-1}&\dots &x_n^{n-1}&-\alpha_n
\end{pmatrix}.
\]
For every $i$, we see that $\gamma_i = (\Gamma \cdot (\partial_1,\ldots,\partial_n,s)^t)_i$.

Let $\Theta$ be the matrix formed by the first $n$ columns of $\Gamma$.
We now find an invertible lower-triangular matrix $M$ such that $\Theta' = M\Theta$ is upper-triangular.

For any finite set of variables $S$ of size at most $k$, let $e_k(S)$ denote the $k$th elementary symmetric polynomial in the elements of $S$:
\[
e_k(S) = \sum_{T\subset S, |T| = k}\left( \prod_{i\in T}x_i \right).
\]
Direct computation by Gaussian elimination shows that 
\[
M =
\begin{pmatrix}
  1&0&\dots&0\\
  m_{21}&1&\dots &0\\
  \vdots &\vdots &\ddots &\vdots\\
  m_{n1}&m_{n2}&\dots &1
\end{pmatrix},\text{ where }
\begin{cases}m_{ij} = (-1)^{i+j}e_{i-j}(S_i)&\text{ for }\\S_{i} = \{x_1,\dots,x_{i-1}\}.\end{cases}
\]

For every $i$ and $j$, one finds the following:
\begin{equation}
  \label{eq:theta-computation}
  \Theta'_{(i,j)} = (M\Theta)_{(i,j)} =
  \begin{cases}
    \prod_{k < i}(x_i-x_k)& \text{if }i \leq j,\\
    0&\text{otherwise}.
  \end{cases}
\end{equation}
Let $\Gamma' = M\Gamma$, and let $\gamma_i' = (\Gamma'\cdot (\partial_1,\ldots,\partial_n)^t)_i$ for every $i$.
Since $M$ is an invertible matrix, $(\gamma_1,\dots,\gamma_n)$ is a regular sequence if and only if $(\gamma_1',\dots,\gamma_n')$ is a regular sequence.

Consider a grading on $\bbC[T^*X][s]$ in which the variables $\partial_1,\dots,\partial_n$, and $s$ have strictly decreasing weights in that order.
The sequence of principal symbols of $(\gamma_1',\dots,\gamma_n')$ under the above grading is just the following sequence formed by the diagonal elements of $\Theta$ times the corresponding $\partial$s, namely $(\Theta'_{(1,1)}\partial_1,\dots,\Theta'_{(n,n)}\partial_n)$.
If this sequence is a regular sequence, then so is the original sequence.

So we now prove that $(\Theta'_{(1,1)}\partial_1,\dots,\Theta'_{(n,n)}\partial_n)$ is a regular sequence.
The first element is just $\partial_1$, which is a non-zerodivisor.
Let $I_m$ be the ideal $(\Theta'_{(1,1)}\partial_1,\dots,\Theta'_{(m-1,m-1)}\partial_{m-1})$.
By induction, it remains to prove that for every $m > 1$, the term $\Theta'_{(m,m)}\partial_m$ is a non-zerodivisor modulo $I_m$.
Recall that
\[
\Theta'_{(m,m)}\cdot \partial_m = \left(\prod_{k < m}(x_m-x_k)\right)\partial_m.
\]
The factor $\partial_m$ of the above expression is one of the generating variables of the polynomial ring $\bbC[T^*X][s]$,
and it does not appear in any of the generators of $I_m$.
Hence it is a non-zerodivisor modulo $I_m$.

Now consider the factor $(x_m - x_{m-1})$.
We can make the following linear change of variables:
\[
z_i =
\begin{cases}
  x_{i+1} - x_i, &1 \leq i \leq (m-1),\\
  x_1 + \dots + x_m, & i = m.
\end{cases}
\]
Then $z_{m-1} = (x_m - x_{m-1})$, and it is easy to check that this variable does not appear in the generators of the ideal $I_m$.
Hence $(x_m-x_{m-1})$ is a non-zerodivisor modulo $I_m$.
For each of the other factors $(x_m - x_i)$, a suitable similar change of variables shows that the variable corresponding to $(x_m - x_i)$ does not appear in the generators of $I_m$.
  
All together, $\Theta'_{(m,m)}\partial_m$ is a non-zerodivisor modulo $I_m$, and the induction argument shows that $(\Theta'_{(1,1)}\partial_1,\dots,\Theta'_{(n,n)}\partial_n)$ is a regular sequence.  
\end{proof}

\begin{lemma}
  \label{lem:regularity-xi}
  The function $\xi_n$ is a non-zerodivisor in $\bbC[T^*X]/(\gamma_1,\dots,\gamma_n)$.
\end{lemma}

\begin{proof}
  Let $J$ be the ideal $(\gamma_1,\dots, \gamma_n)$.
  Since $\xi_n = \prod_{1 \leq i < j\leq n}(x_i-x_j)$, it suffices to show that each linear factor of $\xi_n$ is a non-zerodivisor modulo $J$.
  
  First consider the factor $(x_1-x_2)$.
  If $(x_1-x_2)$ were a zerodivisor modulo $J$, then we would have an equation in $\bbC[T^*X]$ of the following form, for some $f,f_1,\ldots,f_n$ in $\bbC[T^*X]$:
  \[
  (x_1-x_2)\cdot f = \gamma_1f_1 + \dots + \gamma_nf_n.
  \]
  Such an equation can only hold if $\Gamma$ has a non-trivial kernel modulo the ideal $(x_1-x_2)$.
  We now show that the matrix $\Gamma$ has full rank modulo $(x_1 - x_2)$, which will ensure that $(x_1-x_2)$ is a non-zerodivisor modulo $J$.
  Equivalently, we show that the matrix $\overline{\Gamma}$, obtained by deleting the first column of $\Gamma$, has full rank when we set $x_1 = x_2$.

  Recall from the definition of $\Gamma$ that
  \[
  \overline{\Gamma} =
  \begin{pmatrix}
    1&\dots &1 &-\alpha_1\\
    x_2&\dots &x_n&-\alpha_2\\
    \vdots & \ddots & \vdots &\vdots\\
    x_2^{n-1}&\dots &x_n^{n-1}&-\alpha_n
  \end{pmatrix}.
  \]
  We use Gaussian elimination to find the rank of $\overline{\Gamma}$.
  As before, we find an invertible lower-triangular matrix $\overline{M}$ such that $\overline{M}\,\overline{\Gamma}$ is upper-triangular:
  \[
  \overline{M} =
  \begin{pmatrix}
    1&0&\dots&0\\
    \overline{m}_{21}&1&\dots &0\\
    \vdots &\vdots &\ddots &\vdots\\
    \overline{m}_{n1}&\overline{m}_{n2}&\dots &1
  \end{pmatrix},
  \text{ where }
  \begin{cases}\overline{m}_{ij} = (-1)^{i+j}e_{i-j}(\overline{S}_i)&\text{ for }\\ \overline{S}_i = \{x_2,\dots,x_i\}.\end{cases}
  \]
The first $(n-1)$ diagonal entries of $\overline{M}\,\overline{\Gamma}$ are non-zero modulo $(x_1-x_2)$, by a computation similar to \eqref{eq:theta-computation}.

Let $\beta$ denote the last diagonal entry of $\overline{M}\,\overline{\Gamma}$. Then $\beta$ is the following combination of the elements $\{-\alpha_i\}$:
\[
\beta = -\alpha_n - \sum_{i=1}^{n-1}\overline{m}_{ni}\alpha_i.
\]
We consider $\beta$ modulo the ideal $(x_1 - x_2,x_3,\dots, x_n)$.
In other words, we set $x_1 = x_2$, and further set $x_3,\dots,x_n$ equal to zero.
For every $i < (n-1)$, we have $\overline{m}_{ni} \equiv 0$, because every monomial of $e_{n-i}(\lbrace x_2,\dots,x_n\rbrace)$ contains a variable that has been set to zero.
So we find that
\[
\beta \equiv-\alpha_n - \overline{m}_{n,n-1}\alpha_{n-1}.
\]
By \eqref{eqn:alphak}, when $k>1$,
\[
\alpha_k \equiv (2n+k-5) x_2^{k-2}.
\]
Therefore $\beta \equiv - x_2^{n-2}$ modulo the ideal $(x_1-x_2,x_3,\ldots,x_n)$. 

In particular, $\beta\neq 0$.
Since all diagonal entries of $\overline{M}\,\overline{\Gamma}$ are nonzero, it has rank $n$ (full rank) modulo $(x_1-x_2)$, and so $\overline{\Gamma}$ has rank $n$ modulo $(x_1-x_2)$.
Therefore $(x_1-x_2)$ is a non-zerodivisor modulo $J$.
By symmetry, all other factors of $\xi_n$ are non-zerodivisors modulo $J$, so $\xi_n$ is itself a non-zerodivisor modulo $J$.
\end{proof}

We can now put together the preceding results to prove that the $D$-modules generated by the symbols $\xi_n^s$ and $\xi_n^{-s-1}$ are dual.
\begin{proof}[Proof of \autoref{thm:d-module-dual}]
  Recall that $(\gamma_1,\dots,\gamma_n,\xi_n)$ is the sequence of principal symbols of the elements $\{(\delta_1-\alpha_1s), \dots, (\delta_n - \alpha_ns), \xi_n\}$.
  By \autoref{lem:regularity-gammas} and \autoref{lem:regularity-xi}, this is a regular sequence in $\bbC[T^*X][s]$.
  This shows that $V(\xi_n)$ is strongly Koszul, and by \autoref{thm:coxeter-arrangements}, $V(\xi_n)$ is a free divisor.
  Finally, by \autoref{cor:macarro-summary} we conclude that the dual of $\calD_X[s]\xi_n^s$ is isomorphic to $\calD_X[s]\xi_n^{-s-1}$.
\end{proof}

%%%%%%%%%%%%%%%%%%%%%%%%%%%%%%%%%%% 
%%%%%%%%%%%%%%%%%%%%%%%%%%%%%%%%%%%

\section{Factors of \texorpdfstring{$b_{\xi_n}$}{b xi n}}

%%%%%%%%%%%%%%%%%%%%%%%%%%%%%%%%%%%
Recall from the introduction that $\lambda\vdash n$ denotes a partition $\lambda = (\lambda_1\geq \dots \geq \lambda_k)$ of $n$.
Also, $b_\lambda$ denotes a product of the $b$-functions of the functions $\xi_{\lambda_i}$.
In this section, we prove the following.

\begin{proposition}
  \label{thm:bnm1_divides_bn}
  Let $\lambda \vdash n$.
  Then $b_{\xi_{\lambda}}  \mid b_{\xi_{n}}$.
\end{proposition}
We first describe the local $b$-functions of $\xi_n$.

\begin{proposition} \label{prop:localb}
  Let $q = (q_1,\dots,q_n) \in \bbC^n$.
  Let $P$ be any set partition of $\{1,\ldots, n\}$ into subsets such that $i$ and $j$ are in the same subset if and only if $q_i = q_j$.
  Then $b_{\xi_n,q} = b_{\xi_P}$.
\end{proposition}

\begin{proof}
Let 
\[
f = \prod_{\{(i,j)\in P_a\times P_b\mid a\neq b, i < j\}} (x_i - x_j).
\]
Then we can write $\xi_n = f \xi_{P_1} \dots \xi_{P_n}$.

In the stalk $\mathcal{O}_{\bbC^n,q}$, we note that factors $(x_i-x_j)$ where $i$ and $j$ are in different $P_k$ are invertible, and so $f$ is invertible. 
None of the factors of any $\xi_{P_k}$ are invertible, however.  

Note that $\xi_{P_k}$ is just $\xi_{\lambda_k}$ up to renaming coordinates.
Thus they have the same $b$-function.
That is, if $L_k$ is a Bernstein operator for $\xi_{P_k}$, then
\[
L_k \xi_{P_k}^{s+1} = b_{\xi_{\lambda_k}}(s) \xi_{P_k}^s.
\]
Since the $\xi_{P_k}$ are defined in terms of disjoint sets of variables, the operators $L_k$ and polynomials $\xi_{P_l}$ all commute for $l \not = k$.
Define $L = f^s L_{1} \ldots L_r f^{-s-1}$.   
Then
\begin{align*}
L \cdot \xi_n^{s+1} &= 
f^s L_{1} \ldots L_r f^{-s-1} \cdot \xi_{P_1}^{s+1} \cdots \xi_{P_r}^{s+1} f^{s+1}  \\ 
&= f^s \left( L_{1} \cdot \xi_{P_1}^{s+1} \right) \ldots \left( L_r \cdot \xi_{P_r}^{s+1} \right) \\  
&= f^s b_{\xi_{\lambda_1}}(s) \xi_{P_1}^s \ldots b_{\xi_{\lambda_r}}(s) \xi_{P_r}^s  \\
&= b_{\xi_{\lambda_1}}(s) \ldots  b_{\xi_{\lambda_r}}(s) \xi_n^s.
\end{align*}

In order for the above computation to show that $L$ is a local Bernstein operator, we must show that $ L \in \calD_{\bbC^n,q}[s]$. 
Denote $|\bar l| = \sum_{i=1}^n l_i$, and $\bar x^{\bar l} = \prod_{i=1}^n x_i^{l_i}$, and by $(x)_l = x (x-1) \ldots (x-l+1),$ the falling Pochhammer symbol.
Write an expanded form for the operator
\[ 
L_1 \dots L_r = \sum_{i=1}^m c_i \bar x^{\bar j_i} \bar \partial^{\bar k_i}.
\]
Weyl algebra commutation relations then show, 
\begin{align*}
f^s L_1 \dots L_r f^{-s-1} &= f^s \sum_{i=1}^m c_i \left( \sum_{\bar l_i}^{\bar k_i} f^{-s-1-|\bar l_i|} (-s-1)_{|\bar l_i|} f^{(\bar l_i)} \bar \partial^{(\bar l_i - \bar k_i)} \right) \\
&=  \sum_{i=1}^m c_i \left( \sum_{\bar l_i}^{\bar k_i} f^{|\bar l_i|} (-s-1)_{|\bar l_i|} f^{(\bar l_i)} \bar \partial^{(\bar l_i - \bar k_i)} \right) .
\end{align*}
Thus, since $f$ is invertible at $q$, $L \in \calD_{X,q}[s].$

Therefore, we have shown the local $b$-function
\[
b_{\xi_n,q}(s) \mid b_{\xi_{P}}(s).
\]
Since $f$ is invertible and no factor of any $\xi_{P_k}$ is invertible, the local $b$-function $b_{\xi_n,q}(s)$ is a multiple of $b_{\xi_{P_1}\dots \xi_{P_r}}(s) = b_{\xi_P}(s)$.
Thus 
\[
b_{\xi_{P}}(s)  \mid b_{\xi_n,q}(s). 
\] 
So the local $b$-function for $\xi_n$ at $q$ is indeed $b_{\xi_{P}}(s)$. 
\end{proof}

\autoref{thm:bnm1_divides_bn} follows immediately.

\begin{proof}[Proof of \autoref{thm:bnm1_divides_bn}]
Choose $q \in \bbC^n$ such that the set partition $P$ defined as in the hypothesis of \autoref{prop:localb} is the set partition corresponding to $\lambda$. 
Then by \autoref{prop:localb}, $b_{\xi_\lambda}$ is equal to the local $b$-function $b_{\xi_n,q}$ and thus divides $b_{\xi_n}$. 
\end{proof}

%%%%%%%%%%%%%%%%%%%%%%%%%%%%%%%%%%%%%%%%%%%%%%
%%%%%%%%%%%%%%%%%%%%%%%%%%%%%%%%%%%%%%%%%%%%%%

\section{Blow-up computations}

%%%%%%%%%%%%%%%%%%%%%

We blow up $X = \mathbb{C}^n$ along the most singular locus of $V(\xi_n)$, namely,
\[
 Y = \lbrace x_1,\dots,x_n \mid x_1 = \dots = x_n \rbrace. 
\]  
We denote the pullback of $\xi_n$ to the blowup $\widetilde{X}=\bl_Y X$ by $\widetilde{\xi}_n$.
In any affine open subset $U$ of $\widetilde{X}$, it makes sense to compute the $b$-function of $\widetilde{\xi}_n|_U$.
We will denote the least common multiple of all such $b$-functions by $b_{\widetilde{\xi}_n}$, and call it the $b$-function of $\widetilde{\xi}_n$.

Since blowing up makes the variety less singular in some sense, the $b$-function of the blow-up is often easier to compute and simpler than the original divisor.
The precise relationship is given by Kashiwara, who proved the following theorem describing the relationship of $b_{\xi_n}$ to $b_{\tilde \xi_n}$ \cite[Theorem 5.1]{kas:76}.
As a corollary, one deduces that the roots of $b$-functions are negative rational numbers.

\begin{theorem}[Kashiwara, 1976] \label{thm:kashiwarablowup}
  Let $\widetilde{X} = \bl_Y X$.  Let $f$ be a function on $X$ which pulls back to $\widetilde{f}$ on $\widetilde{X}$.
  Then there exists an integer $N > 0$ such that 
  \begin{equation}
    \label{eqn:root-shifted-product}  
    b_f(s)\mid b_{\widetilde{f}}(s)b_{\widetilde{f}}(s+1)\ldots b_{\widetilde{f}}(s+N).
  \end{equation}
\end{theorem}

We will compute $b_{\tilde \xi_n}$, which by the above theorem gives an upper bound on $b_{\xi_n}$.

\begin{proposition} \label{prop:blowup}
  We have the equation
  \begin{equation}\label{eqn:prop-blowup}
    b_{\tilde\xi_n}(s) = \lcm_{\substack{\lambda\vdash n,\\ \lambda \neq n}} \left(b_{\xi_\lambda}(s)\right)
    \cdot \prod_{i=1}^{\binom{n}{2}}\left(s + \frac{i}{\binom{n}{2}} \right).
  \end{equation}
\end{proposition}

\begin{proof}
  The multiplicity of the exceptional divisor in $V(\widetilde{\xi}_n)$ equals $\binom{n}{2}$, namely the number of hyperplanes in $V(\xi_n)$ that intersect at $Y$.
  Locally in an affine chart of $\widetilde{X}$, it is possible to write $\widetilde{\xi}_n$ as $p\cdot q$, where $p$ is the $\binom{n}{2}$th power of a local equation of the exceptional divisor, and $q$ is a function in variables disjoint from those that occur in $p$.
  
  Therefore the $b$-function of $\widetilde{\xi}_n$ is the product of the $b$-functions of $p$ and of $q$.
  Since $V(p)$ is supported on a smooth divisor and has multiplicity $\binom{n}{2}$, we have
  \[
  b_p(s) = \prod_{i=1}^{\binom{n}{2}} \left( s + \frac{i}{\binom{n}{2}} \right).
  \]
  Notice that this is independent of the chart.
  
  We may compute $b_q(s)$ as the least common multiple of the local $b$-functions of $q$ at all points $\widetilde{x}\in V(q)$ away from the exceptional divisor.
  Recall that $\widetilde{X}\to X$ is an isomorphism outside the exceptional divisor.
  Therefore for any $\widetilde{x}$ as above, a local neighborhood around it is isomorphic to a local neighborhood of its image $x\in X$, which is a point of $V(\xi)$ away from $Y$.
  Hence the germ of $\widetilde{\xi}_n$ around $\widetilde{x}$ is isomorphic to the germ of $\xi_n$ around $x$.
  By \autoref{prop:localb}, we see that this germ is (up to a change of coordinates) some $\xi_P$, where $P$ is a non-trivial set partition of $\{1,\dots, n\}$.
  This implies that the local $b$-function of $q$ at $\widetilde{x}$ equals $b_{\xi_P}(s)$.
  All together, $b_q(s)$ is the least common multiple of factors $b_{\xi_P}(s)$, where $P$ ranges over those non-trivial set partitions of $n$ that occur in $V(q)$ away from the exceptional divisor.

  In fact, it is clear that every non-trivial set partition $P$ occurs at some point of $V(q)$ in some local chart.
  Since every $\xi_P$ is isomorphic to some $\xi_\lambda$ up to a permutation of the coordinates, we have $b_{\xi_P}(s) = b_{\xi_\lambda}(s)$ for that $\lambda$.
  Taking the least common multiple of the $b$-functions of $\widetilde{\xi}_n$ on each local chart and over all $P$, we obtain the equation
  \[
  b_{\widetilde{\xi}_n}(s) = \lcm_{\substack{\lambda\vdash n,\\ \lambda \neq n}} \left(b_{\xi_\lambda}(s)\right)
    \cdot \prod_{i=1}^{\binom{n}{2}}\left(s + \frac{i}{\binom{n}{2}} \right).
  \]
\end{proof}

Thus applying \autoref{thm:kashiwarablowup} in the case of \autoref{prop:blowup} we conclude the following.

\begin{corollary} \label{cor:blowupupper}
The $b$-function of $\xi_n$ divides 
\[
\lcm_{\substack{\lambda\vdash n,\\ \lambda \neq n}}\left(b_{\xi_{\lambda}}(s)\right) \cdots \lcm_{\substack{\lambda\vdash n,\\ \lambda \neq n}}\left(b_{\xi_{\lambda}}(s+N)\right)\cdot
\left(s + \frac{1}{\binom{n}{2}} \right) \ldots \left(s + \frac{M}{\binom{n}{2}} \right)
\]
for some $N$ and some $M$ large enough.
\end{corollary}

We conjecture an improved version of \autoref{cor:blowupupper}. 
\begin{conjecture} \label{conj:blowupupper}
The $b$-function of $\xi_n$ divides 
\[
\lcm_{\substack{\lambda\vdash n,\\ \lambda \neq n}}\left(b_{\xi_{\lambda}}(s)\right)\cdot \left(s + \frac{1}{\binom{n}{2}} \right) \ldots \left(s + \frac{M}{\binom{n}{2}} \right)
\]
for some $M$ large enough.
\end{conjecture}

We now outline a possible method towards proving this conjecture.
Let us briefly recall the steps of Kashiwara's proof of \autoref{thm:kashiwarablowup}.
The strategy in this proof is to compare the $\calD_X[s]$-modules $\mathcal{N}_{f} = \calD_X[s]f^s$ and $\mathcal{N}' = \int \calD_X[s] \tilde f^s$.
Kashiwara also constructs a certain cyclic $\calD_X[s]$-module $\mathcal{N}''$ that maps to both of them as follows:
\begin{center}
  \begin{tikzcd}
    \mathcal{N}_{f} \arrow[twoheadleftarrow]{r} & \mathcal{N}'' \arrow[hook]{r} & \mathcal{N}'.
  \end{tikzcd}
\end{center}
In fact, all of the modules above have an action of the shift operator $t$ that shifts all instances of $s$ to $(s+1)$, and hence they are all $\calD[s,t]$-modules.
Let $\mathcal{N}$ be a $\calD[s,t]$-module.
The minimal polynomial for the $s$-action on $(\mathcal{N}/t\mathcal{N})$, if it exists, is called the $b$-function of $\mathcal{N}$, and denoted $b(\mathcal{N})$.
All the $\calD[s,t]$ modules described above have $b$-functions, and they coincide with the Bernstein-Sato $b$-functions $b_f$ and $b_{\widetilde{f}}$ for $\mathcal{N}_f$ and $\calD_X[s]\tilde{f}^s$ respectively.
It is easy to check that $b(\mathcal{N}_f)$ divides $b(\mathcal{N}'')$, and that $b(\mathcal{N}'')$ divides $b_{\widetilde{f}}$.

The equation \eqref{eqn:root-shifted-product} in the proof of \autoref{thm:kashiwarablowup} (in \cite{kas:76}) is the result of comparing the $b$-functions of $\mathcal{N}''$ and $\mathcal{N}'$.
Looking at the difference in the action of $s$ across the inclusion $\mathcal{N}'' \hookrightarrow \mathcal{N}'$ is the reason we cannot write $b_f(s) \mid b_{\widetilde{f}}(s)$ but must instead write a product $b_{\widetilde{f}}(s) \cdots b_{\widetilde{f}}(s+N)$ for the right-hand side.

In the case of $f = \xi_n$, recall that $b_{\tilde \xi_n} = \lcm_{\lambda\vdash n, \lambda \neq n}\left(b_{\xi_\lambda}(s)\right)\cdot c(s)$, where
\[
c(s) = \prod_{i=1}^{\binom{n}{2}}\left(s + \frac{i}{\binom{n}{2}} \right).
\]
From the proof of \autoref{prop:blowup}, we can see that $c(s)$ is the $b$-function of the exceptional divisor, while the other factor comes from outside the exceptional divisor.
Outside the exceptional divisor, the blow-up map is an isomorphism, and thus $\mathcal{N}_{f} \cong \mathcal{N}'$.

It may therefore be possible to decompose $\mathcal{N}'$ into parts corresponding to $c(s)$ and each $b_{\xi_\lambda}(s)$ so that we need only write a product for the $c(s)$ factor, thus obtaining $b_f(s) \mid \lcm_{\lambda\vdash n, \lambda \neq n}\left(b_{\xi_\lambda}(s)\right)\cdot c(s) \cdots c(s+N)$.
%%%%%%%%%%%%%%%%%%%%%%%%%%%%%%%%%%%%%%%%%%%%%%

\section{Partial proof of main conjecture}

The following proposition gives our best upper bound for $b_{\xi_n}$.

\begin{proposition}
  \label{prop:upper-bound}
  The $b$-function of $\xi_n$ divides the following polynomial: 
  \begin{equation}\label{eqn:upper-bound}
\lcm_{\substack{\lambda\vdash n\\ \lambda \neq (n)}}\left(b_{\xi_\lambda}(s)\right) \lcm_{\substack{\lambda\vdash n\\ \lambda \neq (n)}}\left(b_{\xi_\lambda}(s+1)\right) \cdot \prod_{i=(n-1)}^{(n-1)^2}\left(s + \frac{i}{\binom{n}{2}} \right).
  \end{equation}
\end{proposition}

\begin{proof} 

We prove by induction; assume the formula holds for $b_{\xi_{n-1}}(s)$.

There is a relationship between $b$-functions and another singularity invariant, jumping coefficients.
Denote the \emph{multiplier ideal} associated to a function $f$ and a positive rational number $p > 0$ as $\mathcal{I}(f^p)$.
Analytically, this is the ideal in $\mathcal{O}_X$ of functions $h$ for which $|h|^2/|f|^{2p}$ is locally integrable.
(See \cite{ein.laz.smi.ea:04} for an algebro-geometric definition.)
The values of $p$ at which $\mathcal{I}(f^p)$ changes size are called \emph{jumping coefficients}.  
By work of Yano \cite{yan:83}, Lichtin \cite{lic:89}, and Koll\'{a}r \cite[Theorem 10.6]{kol:97}, it is known that the smallest jumping coefficient $\alpha$ satisfies $b_{f}(-\alpha) = 0$, and that in fact $-\alpha$ is the largest root of the $b$-function.

Let $\mathcal{A}$ be a central hyperplane arrangement of $d$ hyperplanes embedded essentially into $\bbC^m$ (i.e. $m$ is minimal).
Let $f$ be a function describing $\mathcal{A}$.
For $S \subset \mathcal{A}$, denote $W_S=\bigcap_{H \in S}H$.
By Musta{\c{t}}{\u{a}} \cite[Corollary 0.3]{mus:06}, the minimal jumping coefficient of $\mathcal{A}$ is 
\[
  \alpha = \min_{\emptyset \not = S  \subset \mathcal{A}}  \frac{\mathrm{codim}\left( W_S\right)}{ \left| \left\lbrace H\mid H \supset W_S     \right\rbrace\right|}.
	\]
Using the above relationship between $b$-functions and jumping coefficients, we can rephrase this (as was done in \cite[Equation 0.2]{sai:06}) as follows:
\begin{equation}\label{eqn:mustata-min-alpha}
\alpha = \min \left( \left\lbrace \frac{m}{d} \right\rbrace \cup \left\lbrace \bigcup_{0 \not = x \in \bbC^m}  \left\lbrace s \mid b_{f,x}(-s)=0 \right\rbrace   \right \rbrace \right).
\end{equation}

The arrangement defined by $\xi_n$ has degree $d =\binom{n}{2}$.
It embeds essentially in $(n-1)$ dimensions via the map $\varphi:\bbC^n \to \bbC^{n-1}$ defined by $y_i = x_i - x_{i+1}$.
In this embedding, the most singular locus $Y$ corresponds to $\{0\}$ alone.  
By \autoref{prop:localb}, for $p\notin Y$, that is, for $p \not = 0$, the local $b$-function equals $b_{\xi_n,p} = b_{\xi_{\lambda_1}} \dots b_{\xi_{\lambda_r}}$ where $\lambda \vdash n$ and $\lambda \not = n$.  
Since $\lambda \not = n$, we have $\lambda_i < n$ for each $i$.
By induction, the minimal $s$ such that $b_{\xi_{\lambda_i}}(-s)=0$ is $(\lambda_i-1)/\binom{\lambda_i}{2}$.
Thus $m/d = (n-1)/\binom{n}{2}$ is the minimal element in \eqref{eqn:mustata-min-alpha}. 
Hence, the roots of $b_{\xi_n}$ are bounded above by $-(n-1)/\binom{n}{2}$.  
By \autoref{cor:symmetry-b}, we see that 
\[ 
\frac{(n-1)}{\binom{n}{2}} - 2 = -\frac{(n-1)^2}{\binom{n}{2}} 
\] 
is a lower bound.  

Thus all the roots of $b_{\xi_n}$ are in the interval $\left[-(n-1)^2/\binom{n}{2}, -(n-1)/\binom{n}{2}\right]$.
We can thus create a tighter upper bound for $b_{\xi_n}$ from \autoref{cor:blowupupper} by removing factors that correspond to roots outside of this interval.
Recall that by \cite[Theorem 1]{sai:06}, the roots of $b_{\xi_k}$ are contained in the interval $(0,2)$ and thus for $\lambda \vdash k$, the roots of $b_{\xi_\lambda}$ are contained in $(0,2)$ as well.
Thus all of $\lcm_{\lambda\vdash n,\lambda \neq (n)}\left(b_{\xi_\lambda}(s+i)\right)$ where $i \geq 2$ lies outside the interval. 
Removing all factors outside this range thus yields \eqref{eqn:upper-bound} as claimed.       
\end{proof}
The Main Conjecture can be proved by means of the above proof provided we replace \autoref{cor:blowupupper} with \autoref{conj:blowupupper}, which is a stronger statement.

%%%%%%%%%%%%%%%%%%%%%%%%%%%%%%%%%%%%%%%%%%%%%%
%%%%%%%%%%%%%%%%%%%%%%%%%%%%%%%%%%%%%%%%%%%%%%
\appendix
\section{Note on irreducibility of function from \texorpdfstring{\cite{wal:14}}{[wal14]}}

This section is an addendum to \cite{wal:14}, a paper by the second named author.  This paper studies a function $f$ defined on $X = M_n(\bbC) \times \bbC^n$ by 
\[
f(M,v) = \det([v\ Mv\ \cdots\ M^{n-1}v]),
\]    
 which is relative invariant with character $\det$.  The paper contains the conjecture that $f$ is irreducible.  It was kindly pointed out by Andr\'{a}s L\H{o}rincz that this conjecture follows from \cite[Proposition 2.3]{wal:14}.

\begin{proposition}[L\H{o}rincz]\label{prop:lorincz}
The polynomial $f$ is irreducible.
\end{proposition}
\begin{proof}
Consider a prime factorization
\begin{equation}\label{eqn:factorization}
 f = \prod_{i=1}^{k} p_i^{m_i}
\end{equation}
 with $m_i$ positive integers.  
By \cite[Proposition 2(2)]{sat:90}, $p_i$ is relative invariant with character $\det^{r_i}$ where $r_i$ is an integer.
Looking at the characters of \autoref{eqn:factorization}, we must have
\[
1 = \sum_{i = 1}^k m_i r_i.
\]  
Thus there is some $i$ such $r_i > 0$.  By \cite[Proposition 2.3]{wal:14}, $p_i = f^{r_i} h$ where $h \in \bbC[X]^{GL_n(\bbC)}$.  Thus $f$ divides $p_i$.  So $f = p_i$ is irreducible.    
\end{proof}

\bibliographystyle{amsalpha}
\bibliography{bibliography}

\end{document}